\documentclass[12pt,twoside]{amsart}
\usepackage{amssymb,amsmath,amsthm, amscd, enumerate, mathrsfs}
\usepackage{graphicx, hhline}
\usepackage[all]{xy}
\usepackage[usenames]{color}
\usepackage{hyperref}
\usepackage{fancyhdr}
\hypersetup{colorlinks=true}

\title{Remarks on the abundance conjecture}
\author{Kenta Hashizume}
\date{2015/10/30, version 0.21}
\address{Department of Mathematics, Graduate School of Science, 
Kyoto University, Kyoto 606-8502, Japan}
\email{hkenta@math.kyoto-u.ac.jp}

\newtheorem{thm}{Theorem}[section]

\newtheorem{cor}[thm]{Corollary}
\newtheorem{prop}[thm]{Proposition}
\newtheorem{conj}[thm]{Conjecture}

\theoremstyle{definition}
\newtheorem{defn}[thm]{Definition}

\newtheorem*{ack}{Acknowledgments} 
\newtheorem{say}[thm]{}

\newtheorem{case}{Case}

\begin{document}

\maketitle
\begin{abstract}
We prove the abundance theorem for log canonical $n$-folds such that the boundary divisor is big 
assuming the abundance conjecture for log canonical $(n-1)$-folds. 
We also discuss the log minimal model program for log canonical $4$-folds.
\end{abstract}

\tableofcontents

\section{Introduction}\label{sec1}
One of the most important open problems in the minimal model theory for higher-dimensional algebraic varieties 
is the abundance conjecture. 
The three-dimensional case of the above conjecture was completely solved (cf.~ \cite{keel-matsuki-mckernan}  
for log canonical threefolds and \cite{fujino1} for semi log canonical threefolds). 
However, Conjecture \ref{conj1.1} is still open in dimension $\geq 4$. 
In this paper, we deal with the abundance conjecture in relative setting.

\begin{conj}[Relative abundance]\label{conj1.1}
Let $\pi:X \to U$ be a projective morphism of varieties and 
$(X,\Delta)$ be a (semi) log canonical pair. 
If $K_{X}+\Delta$ is $\pi$-nef,  
then it is $\pi$-semi-ample.
\end{conj}

Hacon and Xu \cite{haconxu} proved that Conjecture \ref{conj1.1} for log canonical pairs and 
Conjecture \ref{conj1.1} for semi log canonical pairs are equivalent (see also \cite{fujino-gongyo}). 
If $(X, \Delta)$ is Kawamata log terminal and 
$\Delta$ is big, then Conjecture \ref{conj1.1} follows from the usual 
Kawamata--Shokurov base point free theorem in any dimension. 
This special case of Conjecture \ref{conj1.1} plays a crucial role in \cite{bchm}.
Therefore, it is natural to consider Conjecture \ref{conj1.1} for log canonical 
pairs $(X, \Delta)$ under the assumption that 
$\Delta$ is big. 

In this paper, we prove the following theorem.

\begin{thm}[Main Theorem]\label{thm1.1} 
Assume Conjecture \ref{conj1.1} for log canonical $(n-1)$-folds.
Then Conjecture \ref{conj1.1} holds for any projective morphism $\pi :X \to U$ 
and any log canonical $n$-fold $(X, \Delta)$ 
such that $\Delta$ is a $\pi$-big $\mathbb{R}$-Cartier 
$\mathbb{R}$-divisor. 
\end{thm}

We prove it by using the log minimal model program (log MMP, for short) with scaling. 
A key gradient is termination of the log minimal model program with scaling 
for Kawamata log terminal pairs such that the boundary divisor is big (cf.~\cite{bchm}).
For details, see Section \ref{sec3}.

By the above theorem, we obtain the following results in the minimal model theory for $4$-folds.

\begin{thm}[Relative abundance theorem]\label{thm1.2} 
Let $\pi :X \to U$ be a projective morphism from a normal variety to a variety, 
where the dimension of $X$ is four. 
Let $(X, \Delta)$ be a log canonical pair such that $\Delta$ is a $\pi$-big $\mathbb{R}$-Cartier 
$\mathbb{R}$-divisor. 
If $K_{X}+\Delta$ is $\pi$-nef, then it is $\pi$-semi-ample.
\end{thm}

\begin{cor}[Log minimal model program]\label{thm1.8}
Let $\pi :X \to U$ be a projective morphism of normal quasi-projective varieties, 
where the dimension of $X$ is four. 
Let $(X, \Delta)$ be a log canonical pair such that $\Delta$ is a $\pi$-big $\mathbb{R}$-Cartier 
$\mathbb{R}$-divisor.
Then any log MMP of $(X, \Delta)$ with scaling over $U$ terminates with a good minimal model or a 
Mori fiber space of $(X, \Delta)$ over $U$. 
Moreover, if $K_{X}+\Delta$ is $\pi$-pseudo-effective, then any log MMP of $(X,\Delta)$ over $U$ terminates.
\end{cor}

\begin{cor}[Finite generation of adjoint ring]\label{thm1.3}
Let $\pi :X \to U$ be a projective morphism from a normal variety to a variety, 
where the dimension of $X$ is four. 
Let $\Delta^{\bullet}=(\Delta_{1}, \, \cdots , \Delta_{n})$ be an 
$n$-tuple of $\pi$-big $\mathbb{Q}$-Cartier $\mathbb{Q}$-divisors such that 
$(X, \Delta_{i})$ is log canonical for any $1 \leq i \leq n$. 
Then the adjoint ring 
$$\mathcal{R}(\pi, \Delta^{\bullet})
=\underset{(m_{1}, \, \cdots ,\, m_{n}) \in (\mathbb{Z}_{ \geq 0})^{n}}{\bigoplus}
\pi_{*} \mathcal{O}_{X}(\llcorner \sum_{i=1}^{n}m_{i}(K_{X}+\Delta_{i}) \lrcorner)$$ 
is a finitely generated $\mathcal{O}_{U}$-algebra. 
\end{cor}

We note that we need to construct log flips for log canonical pairs 
to run the log minimal model program.
Fortunately, the existence of log flips for log canonical pairs is known for all dimensions 
(cf.~\cite{shokurov-flip} for threefolds, \cite{fujino-fin} for $4$-folds and 
\cite{birkar-flip} or \cite{haconxu-flip} for all higher dimensions).
Therefore we can run the log minimal model program for log canonical pairs in all dimensions. 
By the above corollaries, we can establish almost completely the minimal model theory
for any log canonical $4$-fold $(X,\Delta)$ such that $\Delta$ is big. 

The contents of this paper are as follows. 
In Section \ref{sec2}, we collect some notations and definitions for reader's convenience. 
In Section \ref{sec3}, we prove Theorem \ref{thm1.1}.
In Section \ref{sec4}, we discuss the log minimal model program for log canonical $4$-folds and prove
Theorem \ref{thm1.2}, Corollary \ref{thm1.8} and Corollary \ref{thm1.3}. 

Throughout this paper, we work over the complex number field.

\begin{ack}
The author would like to thank his supervisor Professor Osamu Fujino for many useful advice and suggestions. 
He is grateful to Professor Yoshinori Gongyo for giving information about the latest studies of 
the minimal model theory. 
He also thanks my colleagues for discussions.  
\end{ack}

\section{Notations and definitions}\label{sec2}

In this section, we collect some notations and definitions. 
We will freely use the standard notations in \cite{bchm}. 
Here we write down some important notations and definitions for reader's convenience.

\begin{say}[Divisors]\label{say2.1}
Let $X$ be a normal variety. 
${\rm WDiv}_{\mathbb{R}}(X)$ is the $\mathbb{R}$-vector space with canonical basis given 
by the prime divisors of $X$. 
A variety $X$ is called $\mathbb{Q}${\it -factorial} if every Weil divisor is $\mathbb{Q}$-Cartier.
Let $\pi :X \to U$ be a morphism from a normal variety to a variety and 
let $D=\sum a_{i}D_{i}$ be an $\mathbb{R}$-divisor on $X$. 
Then $D$ is a {\it boundary} $ \mathbb{R}${\it -divisor} if $0 \leq a_{i} \leq 1$ for any $i$.
The {\it round down} of $D$, denoted by $\llcorner D \lrcorner$, is $\sum \llcorner a_{i} \lrcorner D_{i}$ 
where $\llcorner a_{i} \lrcorner$ is the largest integer which is not greater than $a_{i}$. 
$D$ is {\it pseudo-effective over} $U$ (or $\pi${\it -pseudo-effective}) 
if $D$ is $\pi$-numerically equivalent to the limit of effective
$\mathbb{R}$-divisors modulo numerically equivalence over $U$. 
$D$ is  {\em nef over} $U$ (or $\pi${\it -nef}) if it is $\mathbb{R}$-Cartier and  
$\left(D \cdot C\right)\geq 0$ for every proper curve $C$ 
on $X$ contained in a fiber of $\pi$. 
$D$ is {\it big over} $U$ (or $\pi${\it -big}) if it is $\mathbb{R}$-Cartier and 
there exists a $\pi$-ample divisor $A$ 
and an effective divisor $E$ such that $D \sim_{\mathbb{R},\, U} A+E$. 
$D$ is {\it semi-ample over} $U$ (or $\pi${\it -semi-ample}) 
if $D$ is an $\mathbb{R}_{\geq 0}$-linear combination of semi-ample Cartier divisors over $U$, 
or equivalently, there exists a morphism $f: X \to Y$ to a variety $Y$ over $U$ such that 
$D$ is $\mathbb{R}$-linearly equivalent to the pullback of an ample $\mathbb{R}$-divisor over $U$. 
\end{say}

\begin{say}[Singularities of pairs]\label{say2.2} 
Let $\pi :X \to U$ be a projective morphism from a normal variety to a variety and 
$\Delta$ be an effective $\mathbb{R}$-divisor 
such that $K_X+\Delta$ is $\mathbb{R}$-Cartier. 
Let $f:Y\to X$ be a birational morphism. Then $f$ is called a {\it log resolution} of the pair $(X, \Delta)$ 
if $f$ is projective, $Y$ is smooth, the exceptional locus ${\rm Ex}(f)$ is pure codimension one and
${\rm Supp}\,f_{*}^{-1}\Delta \cup {\rm Ex}(f)$ is simple normal crossing.
Suppose that $f$ is a log resolution of the pair $(X, \Delta)$. 
Then we may write 
$$K_Y=f^*(K_X+\Delta)+\sum b_i E_i$$ 
where $E_{i}$ are distinct prime divisors on $Y$. 
Then the {\em log discrepancy} $a(E_{i}, X,\Delta)$ of $E_{i}$ with respect to $(X,\Delta)$ is $1+b_{i}$. 
The pair $(X, \Delta)$ is called {\it Kawamata log terminal} ({\it klt}, for short) if 
$a(E_{i}, X, \Delta) > 0$ for any log resolution $f$ of $(X, \Delta)$ and any $E_{i}$ on $Y$. 
$(X, \Delta)$ is called {\it log canonical} ({\it lc}, for short) if 
$a(E_{i}, X, \Delta) \geq 0$ for any log resolution $f$ of $(X, \Delta)$ and any $E_{i}$ on $Y$. 
$(X, \Delta)$ is called {\it divisorially log terminal} ({\it dlt}, for short) if $\Delta$ is a boundary 
$\mathbb{R}$-divisor and there exists a log resolution 
$f:Y \to X$ of $(X, \Delta)$ such that $a(E, X, \Delta) > 0$ for any $f$-exceptional divisor $E$ on $Y$.
\end{say}

\begin{defn}[log minimal models]\label{defn2.3}
Let $\pi:X \to U$ be a projective morphism from a normal variety to a variety and let 
$(X,\Delta)$ be a log canonical pair.  
Let $\pi ':Y \to U$ be a projective morphism from a normal variety to $U$ 
and $\phi:X \dashrightarrow Y$ 
be a birational map over $U$ such that $\phi^{-1}$ does not contract any divisors. 
Set $\Delta_{Y}=\phi_{*}\Delta$. 
Then the pair $(Y,\Delta_{Y})$ is a {\em log minimal model} of $(X,\Delta)$ over $U$ if 
\begin{itemize}
\item[($1$)]
$K_{Y}+\Delta_{Y}$ is nef over $U$, and 
\item[($2$)]
for any $\phi$-exceptional prime divisor $D$ on $X$, we have
$$a(D, X, \Delta) < a(D, Y, \Delta_{Y}).$$ 
\end{itemize}
A log minimal model $(Y,\Delta_{Y})$ of $(X, \Delta)$ 
over $U$ is called a {\it good minimal model} if $K_{Y}+\Delta_{Y}$ is semi-ample over $U$.    
\end{defn}

Finally, let us recall the definition of semi log canonical pairs.

\begin{defn}[semi log canonical pairs, cf.~{\cite[Definition 4.11.3]{fujino-book}}]\label{defn2.3a}
Let $X$ be a reduced $S_{2}$ scheme. 
We assume that it is pure $n$-dimensional and normal crossing in codimension one. 
Let $X=\cup X_{i}$ be the irreducible decomposition 
and let $\nu: X'=\amalg X_{i}' \to X= \cup X_{i}$ be the normalization. 
Then the {\em conductor ideal} of $X$ is defined by 
$$\mathfrak{cond}_{X}=
\mathcal{H}om_{\mathcal{O}_{X}}(\nu_{*}\mathcal{O}_{X'}, \mathcal{O}_{X})\subset \mathcal{O}_{X}$$
and the {\em conductor} $\mathcal{C}_{X}$ of $X$ is the subscheme defined by $\mathfrak{cond}_{X}$. 
Since $X$ is $S_{2}$ scheme and normal crossing in codimension one, $\mathcal{C}_{X}$ is a reduced 
closed subscheme of pure codimension one in $X$. 

Let $\Delta$ be a boundary $\mathbb{R}$-divisor on $X$ such that $K_{X}+\Delta$ is $\mathbb{R}$-Cartier
and ${\rm Supp}\, \Delta$ does not contain any irreducible component of $\mathcal{C}_{X}$. 
An $\mathbb{R}$-divisor $\Theta$ on $X'$ is defined by $K_{X'}+\Theta=\nu^{*}(K_{X}+\Delta)$ and 
we set $\Theta_{i}=\Theta \!\!\mid _{X_{i}'}$. 
Then $(X, \Delta)$ is called {\em semi log canonical} ({\em slc}, for short) if $(X_{i}', \Theta_{i})$ is lc for any $i$. 
\end{defn}

\section{Proof of the main theorem}\label{sec3}
In this section, we prove Theorem \ref{thm1.1}.
Before the proof, let us recall the useful theorem called dlt blow-up by Hacon. 

\begin{thm}[cf.~{\cite[Theorem 10.4]{fujino-fund}}, {\cite[Theorem 3.1]{kollarkovacs}}]\label{thm4.1}
Let $X$ be a normal quasi-projective variety of dimension $n$ 
and let $\Delta$ be an $\mathbb{R}$-divisor such that $(X, \Delta)$ is log canonical. 
Then there exists a projective birational morphism $f:Y \to X$ from a normal quasi-projective 
variety $Y$ such that 
\begin{itemize}
\item[(1)]
$Y$ is $\mathbb{Q}$-factorial, and
\item[(2)]
if we set
$$\Delta_{Y}=f_{*}^{-1}\Delta +\sum_{E{\rm :}f \mathchar`- {\rm exceptional}}E,$$
then $(Y, \Delta_{Y})$ is dlt and $K_{Y}+\Delta_{Y}=f^{*}(K_{X}+\Delta)$.
\end{itemize}
\end{thm}

\begin{proof}[Proof of Theorem \ref{thm1.1}]
Without loss of generality, we can assume that $U$ is affine. 
By Theorem \ref{thm4.1}, we may assume that $X$ is $\mathbb{Q}$-factorial.
Let $V$ be the finite dimensional subspace in 
${\rm WDiv}_{\mathbb{R}}(X)$ spanned by all components of $\Delta$.
We set
$$\mathcal{N}=\{B \in V \! \mid \! (X,B){\rm \; is\; log\; canonical\; and \;}
K_{X}+B{\rm \;is\;}\pi\mathchar`-{\rm nef}\}.$$
Then $\mathcal{N}$ is a rational polytope in $V$
(cf.~\cite[Theorem 4.7.2 (3)]{fujino-book}, \cite[6.2. First Main Theorem]{shokurov}). 
Since $\Delta$ is $\pi$-big, there are finitely many $\pi$-big $\mathbb{Q}$-Cartier $\mathbb{Q}$-divisors 
$\Delta_{1}, \, \cdots , \Delta_{l} \in \mathcal{N}$ and positive real numbers $r_{1}, \, \cdots, r_{l}$ such that 
$\sum_{i=1}^{l}r_{i}=1$ and $\sum_{i=1}^{l}r_{i}\Delta_{i}=\Delta$. 
Since we have $K_{X}+ \Delta =\sum_{i=1}^{l}r_{i}(K_{X}+\Delta_{i})$, 
it is sufficient to prove that $K_{X}+ \Delta_{i}$ is $\pi$-semi-ample for any $i$. 
Therefore we may assume that $\Delta$ is a $\mathbb{Q}$-divisor.
By using Theorem \ref{thm4.1} again, we may assume that $(X,\Delta)$ is dlt.

If $\llcorner \Delta \lrcorner=0$, then $(X, \Delta)$ is klt and 
Theorem \ref{thm1.1} follows from \cite[Corollary 3.9.2]{bchm}. 
Thus we may assume that $\llcorner \Delta \lrcorner \neq 0$.
Let $k$ be a positive integer such that $k(K_{X}+\Delta)$ is Cartier. 
Pick a sufficiently small positive rational number $\epsilon$ such that 
$\Delta-\epsilon \llcorner \Delta \lrcorner$ is big over $U$ and 
$(2k \epsilon \cdot {\rm dim}\,X)/(1-\epsilon)<1.$
By \cite[Lemma 3.7.5]{bchm}, there is a boundary $\mathbb{Q}$-divisor $\Delta'$, 
which is the sum of a general $\pi$-ample $\mathbb{Q}$-divisor and an effective divisor, such that 
$K_{X}+(\Delta-\epsilon \llcorner \Delta \lrcorner) \sim_{\mathbb{Q},\,U}K_{X}+\Delta'$
and $(X,\Delta')$ is klt. 
By \cite[Theorem E]{bchm}, the $(K_{X}+\Delta')$-log MMP with scaling a $\pi$-ample divisor 
$$X=X_{1}\dashrightarrow X_{2}
\dashrightarrow \cdots \dashrightarrow X_{i}\dashrightarrow \cdots$$
over $U$ terminates.
Since $\Delta-\epsilon \llcorner \Delta \lrcorner \sim_{\mathbb{Q},\,U}\Delta'$, 
it is also the log MMP of $(X,\Delta-\epsilon\llcorner \Delta \lrcorner )$ over $U$.
Let $\Delta_{i}$ be the birational transform of $\Delta$ on $X_{i}$. 
Then $(X_{m},\Delta_{m}-\epsilon \llcorner \Delta_{m} \lrcorner)$ is a log minimal model or 
a Mori fiber space $h:X_{m}\to Z$ of $(X,\Delta-\epsilon\llcorner \Delta \lrcorner )$ 
over $U$ for some $m \in \mathbb{Z}_{> 0}$.

Let $f_{i}:X_{i}\to V_{i}$ be the contraction morphism of the 
$i$-th step of the log MMP 
over $U$, that is, $X_{i+1}= V_{i}$ or $X_{i+1}\to V_{i}$ is the flip of $f_{i}$ over $U$. 
Then $K_{X_{i}}+\Delta_{i}$ is nef over $U$ and $f_{i}$-trivial for any $i\geq1$. 
Indeed, by the induction on $i$, it is sufficient to prove that $K_{X}+\Delta$ is $f_{1}$-trivial 
and $K_{X_{2}}+\Delta_{2}$ is nef over $U$.
Recall that $k$ is a positive integer such that $k(K_{X}+\Delta)$ is Cartier. 
We show that $K_{X}+\Delta$ is $f_{1}$-trivial and 
$k(K_{X_{2}}+\Delta_{2})$ is a nef Cartier divisor over $U$. 
Since $K_{X}+\Delta$ is nef over $U$, for any 
$(K_{X}+\Delta-\epsilon\llcorner \Delta \lrcorner)$-negative extremal ray 
over $U$, it is also a $(K_{X}+\Delta-\llcorner \Delta \lrcorner)$-negative 
extremal ray over $U$. 
Then we can find a rational curve $C$ on $X$ contracted by $f_{1}$ such that 
$0<-(K_{X}+\Delta-\llcorner \Delta \lrcorner)\cdot C \leq 2\,{\rm dim}\,X$
by \cite[Theorem 18.2]{fujino-fund}. 
By the choice of $\epsilon$, we have 
\begin{equation*}
\begin{split}
0&\leq k(K_{X}+\Delta)\cdot C \\
&= \frac{k}{1-\epsilon}\bigl( 
(K_{X}+\Delta-\epsilon\llcorner \Delta \lrcorner)\cdot C
-\epsilon(K_{X}+\Delta-\llcorner \Delta \lrcorner)\cdot C
\bigr) \\
&<\frac{k\epsilon}{1-\epsilon}\cdot2\,{\rm dim}\,X <1.
\end{split}
\end{equation*}
Since $k(K_{X}+\Delta)$ is Cartier, $k(K_{X}+\Delta)\cdot C$ is an integer. 
Then we have $k(K_{X}+\Delta)\cdot C=0$ and thus $K_{X}+\Delta$ 
is $f_{1}$-trivial. 
By the cone theorem (cf.~\cite[Theorem 4.5.2]{fujino-book}), 
there is a Cartier divisor $D$ on $V_{1}$ such that $k(K_{X}+\Delta)\sim f_{1}^{*}D$. 
Since $k(K_{X}+\Delta)$ is nef over $U$, $D$ is also nef over $U$. 
Let $g:W \to X$ and $g':W \to X_{2}$ be a common resolution of $X\dashrightarrow X_{2}$. 
Then $g^{*}(K_{X}+\Delta)=g'^{*}(K_{X_{2}}+\Delta_{2})$ by the negativity lemma 
because $K_{X}+\Delta$ is $f_{1}$-trivial. 
Then $k(K_{X_{2}}+\Delta_{2})$ is the pullback of $D$ and 
therefore it is a nef Cartier divisor over $U$. 
Thus, $K_{X_{i}}+\Delta_{i}$ is nef over $U$ and $f_{i}$-trivial for any $i\geq1$. 

Like above, by taking a common resolution of $X_{i}\dashrightarrow X_{i+1}$
and the negativity lemma, we see that  
$K_{X{i}}+\Delta_{i}$ is semi-ample over $U$ if and only if $K_{X_{i+1}}+\Delta_{i+1}$ 
is semi-ample over $U$ for any $1 \leq i \leq m-1$.
By replacing $(X,\Delta)$ with $(X_{m}, \Delta_{m})$, we may assume that 
$X$ is a log minimal model or a Mori fiber space $h:X\to Z$ of 
$(X,\Delta-\epsilon\llcorner \Delta \lrcorner )$ over $U$. 
We note that after replacing $(X,\Delta)$ with $(X_{m}, \Delta_{m})$, $(X,\Delta)$ is lc but not necessarily dlt.

\begin{case}\label{case3.2.1}
$X$ is a Mori fiber space $h:X\to Z$ of 
$(X,\Delta-\epsilon\llcorner \Delta \lrcorner )$ over $U$. 
\end{case}

\begin{proof}[Proof of Case 1]
First, note that in this case $K_{X}+\Delta$ is $h$-trivial by the above discussion. 
Moreover $\llcorner \Delta \lrcorner$ is ample over $Z$.
By the cone theorem (cf.~\cite[Theorem 4.5.2]{fujino-book}), there exists a 
$\mathbb{Q}$-Cartier $\mathbb{Q}$-divisor $\Xi$ on $Z$ such that 
$K_{X}+\Delta \sim _{\mathbb{Q},\, U} h^{*}\Xi$. 
Since $\llcorner \Delta \lrcorner$ is ample over $Z$, ${\rm Supp}\, \llcorner \Delta \lrcorner$ dominates $Z$. 
In particular, there exists a component of  $\llcorner \Delta \lrcorner$, which we put $T$, 
such that $T$ dominates $Z$. 
Let $f:(Y,\Delta_{Y})\to (X,\Delta)$ be a dlt blow-up (see Lemma \ref{thm4.1}) and  
$\widetilde{T}$ be the strict transform of $T$ on $Y$.
Then $K_{X}+\Delta$ is semi-ample over $U$ if and only if $K_{Y}+\Delta_{Y}$ 
is semi-ample over $U$.
Furthermore, we have 
$K_{\widetilde{T}}+{\rm Diff}(\Delta_{Y}- \widetilde{T})=((h\circ f)\! \mid _{\widetilde{T}})^{*}\Xi$ 
since ${\widetilde{T}}$ dominates $Z$.
Thus it is sufficient to prove that 
$K_{\widetilde{T}}+{\rm Diff}(\Delta_{Y}- \widetilde{T})$ is semi-ample over $U$. 
Since $(Y,\Delta_{Y})$ is dlt, $\widetilde{T}$ is normal by \cite[Corollary 5.52]{kollar-mori}. 
By \cite[17.2. Theorem]{kollar}, we see that the pair $(\widetilde{T}, {\rm Diff}(\Delta_{Y}- \widetilde{T}))$ is lc. 
Then $K_{\widetilde{T}}+{\rm Diff}(\Delta_{Y}- \widetilde{T})$ is semi-ample over $U$ by 
the relative abundance theorem for log canonical $(n-1)$-folds.
So we are done.
\end{proof}

\begin{case}\label{case3.2.2}
$X$ is a log minimal model of 
$(X,\Delta-\epsilon\llcorner \Delta \lrcorner )$ over $U$. 
\end{case}

\begin{proof}[Proof of Case 2]
In this case, both $K_{X}+\Delta$ and $K_{X}+\Delta-\epsilon \llcorner \Delta \lrcorner$ are nef over $U$. 
By \cite[Corollary 3.9.2]{bchm},
$K_{X}+\Delta-\epsilon \llcorner \Delta \lrcorner$ is semi-ample over $U$. 
Therefore we may assume that $\llcorner \Delta \lrcorner \neq 0$, and 
there exists a sufficiently large and divisible positive integer $l$ such that 
both $l(K_{X}+\Delta)$ and $l \epsilon \llcorner \Delta \lrcorner$ are Cartier and
$$\pi^{*}\pi_{*}\mathcal{O}_{X}(l(K_{X}+\Delta-\epsilon \llcorner \Delta \lrcorner))
\to \mathcal{O}_{X}(l(K_{X}+\Delta-\epsilon \llcorner \Delta \lrcorner)) $$
is surjective. 
Then, in the following diagram, 
$$
\xymatrix{
\pi^{*}\pi_{*}\mathcal{O}_{X}(l(K_{X}+\Delta-\epsilon \llcorner \Delta \lrcorner))\!
\mid _{X \setminus \llcorner \Delta \lrcorner} \ar[d] \ar[r]&
\pi^{*}\pi_{*}\mathcal{O}_{X}(l(K_{X}+\Delta))\!\mid _{X \setminus \llcorner \Delta \lrcorner} \ar[d] \\
\mathcal{O}_{X}(l(K_{X}+\Delta-\epsilon \llcorner \Delta \lrcorner))\!
\mid _{X \setminus \llcorner \Delta \lrcorner} \ar[r]^{\quad  \;\; \cong}&
\mathcal{O}_{X}(l(K_{X}+\Delta))\!\mid _{X \setminus \llcorner \Delta \lrcorner}
}
$$
the left vertical morphism is surjective. 
Moreover, the lower horizontal morphism is an isomorphism.
Therefore the right vertical morphism is surjective. 
Thus 
 $\pi^{*}\pi_{*}\mathcal{O}_{X}(l(K_{X}+\Delta)) \to \mathcal{O}_{X}(l(K_{X}+\Delta))$ is 
surjective outside of $\llcorner \Delta \lrcorner$.

Next, set $D={\rm Diff}(\Delta-\llcorner \Delta \lrcorner)$ and consider the following exact sequence
\begin{equation*}
\begin{split}
0 \to \mathcal{O}_{X}(l'(K_{X}+\Delta)-\llcorner \Delta \lrcorner) 
&\to \mathcal{O}_{X}(l'(K_{X}+\Delta))\\
&\to \mathcal{O}_{\llcorner \Delta \lrcorner}(l'(K_{\llcorner \Delta \lrcorner}+D)) \to 0,
\end{split}
\end{equation*}
where $l'$ is a sufficiently large and divisible positive integer such that $1/l' \leq \epsilon$. 
Then we have
$$l'(K_{X}+\Delta)-\llcorner \Delta \lrcorner=l'(K_{X}+\Delta-\frac{1}{l'}\llcorner \Delta \lrcorner ).$$
Moreover, $\Delta- (1/l')\llcorner \Delta \lrcorner$ is big over $U$ and 
$K_{X}+ \Delta- (1/l')\llcorner \Delta \lrcorner$ is nef over $U$.
Since $(X, \Delta-(1/l')\llcorner \Delta \lrcorner)$ is klt, by \cite[Lemma 3.7.5]{bchm}, 
we may find a $\pi$-big $\mathbb{Q}$-Cartier $\mathbb{Q}$-divisor $A+B$, 
where $A \geq 0$ is a general ample $\mathbb{Q}$-divisor over $U$ and $B \geq 0$, 
such that $(X, A+B)$ is klt and $\Delta- (1/l')\llcorner \Delta \lrcorner \sim_{\mathbb{Q}, \, U} A+B$. 
In particular, $(X, B)$ is klt. 
Furthermore, $A+(l'-1)(K_{X}+\Delta- (1/l')\llcorner \Delta \lrcorner)$ is ample over $U$.
Thus we have
$$l'(K_{X}+\Delta)-\llcorner \Delta \lrcorner \sim_{\mathbb{Q}, \, U} 
K_{X}+A+(l'-1)(K_{X}+\Delta- \frac{1}{l'}\llcorner \Delta \lrcorner)+B$$
and $R^{1}\pi_{*}\mathcal{O}_{X}(l'(K_{X}+\Delta)-\llcorner \Delta \lrcorner)=0$ 
(cf.~\cite[Theorem 1-2-5]{kawamata-matsuda-matsuki}). 
Then $\pi_{*} \mathcal{O}_{X}(l'(K_{X}+\Delta)) \to 
\pi_{*}\mathcal{O}_{\llcorner \Delta \lrcorner}(l'(K_{\llcorner \Delta \lrcorner}+D))$ is surjective
and thus
$\pi^{*} \pi_{*} \mathcal{O}_{X}(l'(K_{X}+\Delta))\otimes \mathcal{O}_{\llcorner \Delta \lrcorner} \to 
\pi^{*}\pi_{*}
\mathcal{O}_{\llcorner \Delta \lrcorner}(l'(K_{\llcorner \Delta \lrcorner}+D))$ is surjective. 

We can check that the pair $(\llcorner \Delta \lrcorner, D)$ is semi log canonical. 
Indeed, since $(X, \Delta-\epsilon \llcorner \Delta \lrcorner)$ is klt and since $X$ is $\mathbb{Q}$-factorial, by 
\cite[Corollary 5.25]{kollar-mori}, $\llcorner \Delta \lrcorner$ is Cohen--Macaulay. 
In particular, $\llcorner \Delta \lrcorner$ satisfies the $S_{2}$ condition. 
Moreover, since $(X, \Delta)$ is lc,
$\llcorner \Delta \lrcorner$ is normal crossing in codimension one. 
We also see that $D$ does not contain any irreducible component of 
$\mathcal{C}_{\llcorner \Delta \lrcorner}$ by \cite[16.6 Proposition]{corti}. 
Therefore $(\llcorner \Delta \lrcorner, D)$ is semi log canonical by \cite[17.2 Theorem]{kollar}. 
Since $K_{\llcorner \Delta \lrcorner}+D=(K_{X}+\Delta)\!\! \mid _{\llcorner \Delta \lrcorner}$ is nef over $U$, 
$K_{\llcorner \Delta \lrcorner}+D$ is semi-ample over $U$ by \cite[Theorem 1.4]{haconxu} and 
the relative abundance theorem for log canonical $(n-1)$-folds.

By these facts, in the following diagram, 
$$
\xymatrix{
\pi^{*} \pi_{*} \mathcal{O}_{X}(l'(K_{X}+\Delta))\otimes \mathcal{O}_{\llcorner \Delta \lrcorner}\ar[d] \ar[r]&
\pi^{*}\pi_{*}
\mathcal{O}_{\llcorner \Delta \lrcorner}(l'(K_{\llcorner \Delta \lrcorner}+D))\ar[d] \\
\mathcal{O}_{X}(l'(K_{X}+\Delta))\otimes \mathcal{O}_{\llcorner \Delta \lrcorner}\ar[r]^{\quad \cong}&
\mathcal{O}_{\llcorner \Delta \lrcorner}(l'(K_{\llcorner \Delta \lrcorner}+D))
}
$$
the right vertical morphism and the upper horizontal morphism are both surjective. 
Furthermore, the lower horizontal morphism is an isomorphism. 
Therefore the left vertical morphism is surjective. 
Then $\pi^{*}\pi_{*}\mathcal{O}_{X}(l'(K_{X}+\Delta)) \to \mathcal{O}_{X}(l'(K_{X}+\Delta))$ 
is surjective in a neighborhood of $\llcorner \Delta \lrcorner$. 

Therefore, $\pi^{*}\pi_{*}\mathcal{O}_{X}(l(K_{X}+\Delta)) \to \mathcal{O}_{X}(l(K_{X}+\Delta))$ 
is surjective for some sufficiently large and divisible positive integer $l$.
So we are done. 
\end{proof}
Thus, in both case, $K_{X}+\Delta$ is semi-ample over $U$. 
Therefore we complete the proof.
\end{proof}

\section{Minimal model program in dimension four}\label{sec4}

In this section, we discuss the log minimal model for log canonical $4$-folds and 
prove Theorem \ref{thm1.2}, Corollary \ref{thm1.8} and Corollary \ref{thm1.3}. 

\begin{proof}[Proof of Theorem \ref{thm1.2}]
It immediately follows from Theorem \ref{thm1.1} since the relative abundance theorem for log canonical 
3-folds holds (cf.~\cite[Theorem A.2]{fujino1}). 
\end{proof}

\begin{prop}\label{prop4.1}
Let $\pi :X \to U$ be a projective morphism of normal quasi-projective varieties, 
where the dimension of $X$ is four. 
Let $(X, \Delta)$ be a log canonical pair and let $A$ be an effective $\mathbb{R}$-divisor 
such that $(X, \Delta+A)$ is log canonical and $K_{X}+\Delta+A$ is $\pi$-nef. 
Then we can run the log MMP of $(X,\Delta)$ with scaling $A$ over $U$ and this log MMP with scaling terminates. 
\end{prop}

\begin{proof}
We can run the log MMP of $(X,\Delta)$ with scaling $A$ over $U$ by \cite[Remark 4.9.2]{fujino-book}. 
Therefore we only have to prove the termination of the log MMP with scaling. 

Suppose by contradiction that we get an infinite sequence of birational maps by running 
the log MMP with scaling $A$
$$(X=X_{1},\Delta=\Delta_{1},\lambda_{1}A_{1})
\dashrightarrow \cdots \dashrightarrow(X_{i},\Delta_{i},\lambda_{i}A_{i}) \dashrightarrow \cdots$$
over $U$, where $A_{i}$ is the birational transform of $A$ on $X_{i}$ and
$$\lambda_{i}={\rm inf}\{\mu \in \mathbb{R}_{\geq0} \!\mid 
K_{X_{i}}+\Delta_{i}+\mu A_{i} \;{\rm is}\;{\rm nef\; over}\;U \}$$
for every $i \geq1$. 
Let $X_{i}\to V_{i}$ be the contraction morphism of the $i$-th step of the $(K_{X}+\Delta)$-log MMP with 
scaling $A$ over $U$. 
Note that by \cite[Lemma 3.8]{birkar1}, 
the log MMP with scaling terminates for all $\mathbb{Q}$-factorial dlt $4$-folds. 
By the same argument as in the proof of \cite[Lemma 4.9.3]{fujino-book}, 
we obtain the following diagram
$$
\xymatrix@C=15pt{
(Y_{1}^{1},\Psi_{1}^{1}) \ar_{\alpha_{1}}[d] \ar@{-->}[r]&
\cdots \ar@{-->}[r]&(Y_{1}^{k_{1}}=Y_{2}^{1},\Psi_{1}^{k_{1}}=\Psi_{2}^{1})\ar_{\alpha_{2}}[d] \ar@{-->}[r]&
\cdots\ar@{-->}[r]& (Y_{i}^{1},\Psi_{i}^{1}) \ar_{\alpha_{i}}[d] \ar@{-->}[r]&
\cdots \\
(X_{1}, \Delta_{1}) \ar@{-->}[rr]&&
(X_{2}, \Delta_{2}) \ar@{-->}[r]&
\cdots\ar@{-->}[r]&(X_{i}, \Delta_{i})\ar@{-->}[r]&
\cdots
}
$$
such that 
\begin{enumerate}
\item[(1)]
$(Y_{i}^{1},\Psi_{i}^{1})$ is $\mathbb{Q}$-factorial dlt and 
$K_{Y_{i}^{1}}+\Psi_{i}^{1}=\alpha_{i}^{*}(K_{X_{i}}+\Delta_{i})$, 
\item[(2)]
the sequence of birational maps
$$(Y_{i}^{1},\Psi_{i}^{1})\dashrightarrow\cdots\dashrightarrow
(Y_{i}^{k_{i}},\Psi_{i}^{k_{i}})=(Y_{i+1}^{1},\Psi_{i+1}^{1})$$ 
is a finite number of steps of the $(K_{Y_{i}^{1}}+\Psi_{i}^{1})$-log MMP over $V_{i}$
\end{enumerate}
for any $i\geq1$, and  
\begin{enumerate}
\item[(3)]
the sequence of the upper horizontal birational maps is an
infinite sequence of divisorial contractions and log flips of the $(K_{Y_{1}^{1}}+\Psi_{1}^{1})$-log MMP over $U$.
\end{enumerate}
For every $i \geq1$ and $1\leq j <k_{i}$, let $A_{i}^{j}$ be the birational transform of $\alpha_{1}^{*}A$ on 
$Y_{i}^{j}$ and let 
$$\lambda_{i}^{j}={\rm inf}\{\mu \in \mathbb{R}_{\geq0} \!\mid 
K_{Y_{i}^{j}}+\Psi_{i}^{j}+\mu A_{i}^{j} \;{\rm is}\;{\rm nef\; over}\;U \}.$$
Then we have $\lambda_{i}^{j}=\lambda_{i}$ for any $i \geq1$ and $1\leq j<k_{i}$. 
Indeed, since $K_{X_{i}}+\Delta_{i}+\lambda_{i} A_{i}$ is nef over $U$ and it is also trivial over $V_{i}$, 
there is an $\mathbb{R}$-Cartier divisor $D$, which is nef over $U$, on $V_{i}$ such that 
$K_{X_{i}}+\Delta_{i}+\lambda_{i} A_{i}$ is $\mathbb{R}$-linearly equivalent 
to the pullback of $D$. 
Since $A_{i}^{1}=\alpha_{i}^{*}A_{i}$, by the condition (1), $K_{Y_{i}^{1}}+\Psi_{i}^{1}+\lambda_{i} A_{i}^{1}$ is also 
$\mathbb{R}$-linearly equivalent to the pullback of $D$. 
Thus $K_{Y_{i}^{1}}+\Psi_{i}^{1}+\lambda_{i} A_{i}^{1}$ is nef over $U$.
Moreover, by the condition $(2)$, $K_{Y_{i}^{j}}+\Psi_{i}^{j}+\lambda_{i} A_{i}^{j}$ is also 
$\mathbb{R}$-linearly equivalent to the pullback of $D$. 
Therefore $K_{Y_{i}^{j}}+\Psi_{i}^{j}+\lambda_{i} A_{i}^{j}$ is nef over $U$ and trivial over $V_{i}$ for any $0\leq j<k_{i}$. 
We also see that $K_{Y_{i}^{j}}+\Psi_{i}^{j}+\mu A_{i}^{j}$ is not nef over $V_{i}$ for any $\mu \in [0,\lambda_{i})$ 
by the condition $(2)$. 
In particular it is not nef over $U$. 
Therefore we have $\lambda_{i}^{j}=\lambda_{i}$ for any $i \geq1$ and $1\leq j<k_{i}$.

By these facts, we can identify the sequence of birational maps 
$$(Y_{1}^{1},\Psi_{1}^{1})\dashrightarrow\cdots \dashrightarrow (Y_{i}^{j},\Psi_{i}^{j})\dashrightarrow
(Y_{i}^{j+1},\Psi_{i}^{j+1})\dashrightarrow \cdots$$
with an infinite sequence of birational maps of the $(K_{Y_{1}^{1}}+\Psi_{1}^{1})$-log MMP with scaling $A_{1}^{1}=\alpha_{1}^{*}A$ over $U$. 
But then it must terminate by \cite[Lemma 3.8]{birkar1}.
It contradicts to our assumption. 
So we are done. 
\end{proof}

\begin{proof}[Proof of Corollary \ref{thm1.8}]
The first half of the assertions immediately 
follows from Proposition \ref{prop4.1} and Theorem \ref{thm1.2}. 
For the latter half, if $K_{X}+\Delta$ is $\pi$-pseudo-effective 
then it is $\pi$-effective by the first half of this corollary. 
By \cite[Main Theorem 1.3]{birkar0}, termination of any log MMP follows. 
So we are done. 
\end{proof}

\begin{proof}[Proof of Corollary \ref{thm1.3}]
Without loss of generality, we can assume that $U$ is affine.
Then the assertion follows from Proposition \ref{prop4.1} and Theorem \ref{thm1.2} 
with the same argument as in the proof of 
\cite[Lemma 3.2]{hashizume} and the discussion of \cite[Section 4]{hashizume}.
\end{proof}


\end{document}